\documentclass[12pt]{amsart}
\usepackage{amssymb}
\usepackage{amsmath}
\usepackage{graphicx} 
\usepackage{color}[dvipdfmx] 
\usepackage[all]{xy}
\newcommand{\C}{\mathbb C}
\newcommand{\D}{\mathbb D}

\newcommand{\N}{\mathbb N}
\newcommand{\R}{\mathbb R}

\newcommand{\Z}{\mathbb Z}

\newcommand{\cF}{\mathcal F}

\newcommand{\re}{\mathrm{Re}}

\newcommand{\pa}{\partial}

\renewcommand{\a}{\alpha} 
\renewcommand{\b}{\beta} 

\newcommand{\G}{\varGamma}
\newcommand{\De}{\mathit{\Delta}}

\newcommand{\f}{\varphi}

\newcommand{\z}{\zeta}

\newcommand{\la}{\langle}
\newcommand{\ra}{\rangle}
\newcommand{\tr}{\;^t}

\newcommand{\one}{\mathbf{1}}
\newcommand{\zero}{\mathbf{0}}

\newcommand{\comment}[1]{{}}

\newcommand{\wt}[1]{\widetilde{#1}}
\newcommand{\init}{\operatorname{in}}
\newcommand{\rank}{\operatorname{rank}}
\newcommand{\ord}{\operatorname{ord}}
\newcommand{\sing}{\operatorname{Sing}}
\newcommand{\ch}{\operatorname{Ch}}
\newcommand{\pr}{\operatorname{pr}}

\newcommand{\gr}{\operatorname{gr}}

\newcommand{\A}{{\operatorname{A}}}

\newcommand{\red}[1]{\textcolor{red}{#1}}

\newtheorem{theorem}{Theorem}[section]
\newtheorem{proposition}{Proposition}[section]
\newtheorem{lemma}{Lemma}[section]

\newtheorem{fact}{Fact}[section]
\newtheorem{remark}{Remark}[section]

\newtheorem{definition}{Definition}[section]

\newcommand{\HGF}[5]
{{}_{#1}F_{#2}\left(\begin{matrix}#3\\#4\end{matrix};#5\right)}

\title
[A hypergeometric system of rank $p^m$]
{A system of hypergeometric differential equations in $m$ variables 
 of rank $p^m$}

\makeatletter
\@addtoreset{equation}{section}
\@addtoreset{figure}{section}
\@addtoreset{table}{section}
\makeatother

\author[Kaneko J.]{Kaneko Jyoichi}
\address[Kaneko]{
Department of Mathematical Sciences,
University of the Ryukyus,
Nishihara, Okinawa, 903-0213, Japan
}
\email{kaneko@math.u-ryukyu.ac.jp}

\author[Matsumoto K.]{Matsumoto Keiji}
\address[Matsumoto]{
Department of Mathematics\\
Hokkaido University\\
Sapporo 060-0810, Japan
}
\email{matsu@math.sci.hokudai.ac.jp}

\author[Ohara K.]{Ohara Katsuyoshi}
\address[Ohara]{
Faculty of Mathematics and Physics\\
Kanazawa University\\
Kanazawa 920-1192, Japan\\
}
\email{ohara@se.kanazawa-u.ac.jp
}

\author[Terasoma T.]{Terasoma Tomohide}
\address[Terasoma]{
Faculty of Science and Engineering,
Hosei University,
Koganei, Tokyo 184-8584, Japan
}
\email{terasoma@hosei.ac.jp}


\date{\today}

\begin{document}
\maketitle
\begin{abstract}
We define a hypergeometric series in $m$ variables with $p+(p-1)m$ 
parameters, which reduces to the generalized 
hypergeometric series $_pF_{p-1}$ when $m=1$, and to 
Lauricella's hypergeometric series $F_C$ in $m$ variables when $p=2$. 
We give a system of hypergeometric differential equations annihilating 
the series.  
Under some non-integral conditions on parameters,  
we give an Euler type integral representation of the series, and  
linearly independent $p^m$ solutions to this system around 
a point near to the origin.  
We show that this system is of rank $p^m$, and determine its singular locus.

\end{abstract}
\tableofcontents

\section{Introduction}
The hypergeometric series is defined by 
$$_2F_1\left(\begin{matrix}a_1,a_2\\ b_1
\end{matrix}
;x\right)=\sum_{n=0}^\infty\frac{(a_1,n)(a_2,n)}{(b_1,n)(1,n)}x^n,$$
where $a_1,a_2,b_1$ are complex parameters with 
$b_1\notin -\N=\{0,-1,-2,\dots\}$, and 
$(a_1,n)$ denotes Pochhammer's symbol, that is  
$$(a_1,n)=a_1(a_1+1)\cdots (a_1+n-1)=\frac{\G(a_1+n)}{\G(a_1)}.$$ 
This series converges for $|x|<1$, admits an Euler type integral 
$$
\frac{\G(b_1)}{\G(a_1)\G(b_1-a_1)}
\int_0^1 t^{a_1-1}(1-t)^{b_1-a_1-1}(1-tx)^{-a_2}dt
$$
under $0<\re(a_1)<\re(b_1)$, 
and satisfies the hypergeometric differential equation 
$$(x(a_1+\theta)(a_2+\theta)-\theta(b_1-1+\theta))\cdot f(x)=0,$$
of second order with regular singular points at $x=0,1,\infty$,  
where $f(x)$ is an unknown function and $\theta$ denotes 
the Euler operator $x\frac{d}{dx}$. If the parameter $b_1$ is not an integer, then 
the functions 
$_2F_1\left(\begin{matrix}a_1,a_2\\ b_1\end{matrix}
;x\right)$ and $x^{1-b_1}\;_2F_1\left(\begin{matrix}a_1-b_1+1,a_2-b_1+1\\ 2-b_1
\end{matrix};x\right)$
span the space of solutions to the hypergeometric differential equation 
around a point near to the origin. 
The hypergeometric series and the hypergeometric differential 
equation appear in various fields in  mathematics and physics, 
and play important role. 

In this paper, we define a hypergeometric series $F_C^{p,m}(a,B;x)$ in 
$m$ variables  $x_1$, $\dots$, $x_m$ with parameters $a=\tr(a_1,\dots,a_p)$ 
and  $B=(b_{i,j})_{\substack{{1\le i\le p}\\ {1 \le j\le m}}}$ assigning $b_{p,k}=1$ 
 $(1\le k\le m)$ as in \eqref{eq:HGS}. 
This series reduces to the generalized hypergeometric series 
$$_{p}F_{p-1}\left(\begin{matrix}a_1,\dots,a_p\\
 b_{1,1}\dots,b_{p-1,1}\end{matrix};x_1\right)
=\sum_{n=0}^\infty 
\frac{(a_1,n)\cdots (a_p,n)}{(b_{1,1},n)\cdots (b_{p-1,1},n)(1,n)}x_1^n
$$ when $m=1$, and to Lauricella's hypergeometric series 
\begin{align*}
&F_C(a_1,a_2,b_{1,1},\dots,b_{1,m};x_1,\dots,x_m)\\
=&\sum_{(n_1,\dots,n_m)\in \N^m}
\frac{(a_1,n_1+\cdots+n_m) (a_2,n_1+\cdots+n_m)}
{(b_{1,1},n_1)\cdots (b_{1,m},n_m)(1,n_1)\cdots (1,n_m)}
x_1^{n_1}\cdots x_m^{n_m}
\end{align*}
in $m$ variables $x_1,\dots,x_m$  when $p=2$, where $\N=\{0,1,2,\dots\}$. 
We can regard 
$F_C^{p,m}(a,B;x)$ as one of multi-variable models of the generalized 
hypergeometric series $_pF_{p-1}$ just like Lauricella's $F_C$ is 
that of the hypergeometric series $_2F_1$. 
We also remark that $F_C^{p,m}(a,B;x)$ for $m=2$ coincides with 
one of the generalizations of the hypergeometric series introduced by Kamp\'e
de F\'eriet  as mentioned in \cite[\S XLVII]{AK} 
and \cite[\S 1.5]{Ex}, 
and there are detailed studies on $F_C^{p,m}(a,B;x)$ 
for $(p,m)=(3, 2)$ in \cite{KMO1} and \cite{KMO2}.

We show that the series $F_C^{p,m}(a,B;x)$ converges absolutely on any point in 
$$\D=\{x\in \C^m\mid \sqrt[p]{|x_1|}+\cdots+\sqrt[p]{|x_m|}<1\}.$$
Under some non-integral conditions on the parameters $a,B$,  
we give an Euler type $(p-1)\cdot m$-ple integral for $F_C^{p,m}(a,B;x)$ 
in Theorem \ref{th:int-rep}, which can be regarded as 
a generalization of Euler type $(p-1)$-ple and $m$-ple integrals 
for the generalized hypergeometric series $_{p}F_{p-1}$ and 
for Lauricella's $F_C$, respectively, refer to  \cite[(4.1.3)]{Sl} and 
\cite[Proposition 2.3]{G1} for their explicit forms. 

We find differential operators 
\begin{align*}
\ell_k=&\theta_k(b_{1,k}-1+\theta_k)\cdots (b_{p-1,k}-1+\theta_k)\\
&-x_k(a_{1}+\theta_1+\cdots+\theta_m)\cdots (a_{p}+\theta_1+\cdots+\theta_m)
\quad (k=1,\dots,m)
\end{align*}
annihilating $F_C^{p,m}(a,B;x)$, and define a system 
$\cF_C^{p,m}(a,B)$ of differential equations generated by them, 
where $\theta_k$ denotes the Euler operator $x_k\pa_k$, where 
$\pa_k=\frac{\pa}{\pa x_k}$.

Under some non-integral conditions on the parameters $a,B$, 
we give linearly independent $p^m$ multi-valued solutions of a form
$$
\big(\prod_{k=1}^{m} x_k^{\mu_k}\big)\cdot F_C^{p,m}(a',B';x),
$$ 
around the origin in Proposition \ref{prop:fund-solutions}.
Here $\mu_k=1-b_{j_k,k}$ $(1\le j_k\le p)$ and 
the entries of $a',B'$ are expressed in terms of those of the original $a,B$.   
In Theorem \ref{th:rank},
we prove that the system $\cF_C^{p,m}(a,B)$ 
is of rank $p^m$, as a consequence, the above $p^m$-independent
solutions actually form a basis of the space of solutions.

We state in Theorem \ref{th:sing-loc} that the singular locus is
given by
$$
S(x)=\{x\in \C^m\mid x_1\cdots x_m \cdot R(x)=0\},
$$
where $R(x)$ is a polynomial in $x_1,\dots,x_m$ of degree $p^{m-1}$
defined in \eqref{eq:R(x)}.  
In the paper [HT], they show this theorem for $p=2$. 
To prove these theorems, 
we consider the left 
ideal $\mathcal I$ of the Weyl algebra 
$D_m=\C\langle x_1, \dots, x_m,\pa_1,\dots, \pa_m\rangle$ 
generated by 
$\ell_1,\dots,\ell_m$.
To study the rank and singular locus
in $(\C^\times)^m=\{(x_1, \dots, x_m)\mid x_1\cdots x_m\neq 0\}$,
it is convenient to consider the pull back system
by the unramified covering 
$$
\f:{(\C^\times)}^m\!\ni\! z\!=\!(z_1, \dots, z_m)\mapsto 
x\!=\!(x_1, \dots, x_m)\!=\!(z_1^p,\dots, z_m^p)
\!\in\!{(\C^\times)}^m.
$$
We set $D_m^*=D_m[1/x_1,\dots,1/x_m]$ and 
$\widetilde{D}_m^*=\widetilde{D}_m[1/z_1,\dots,1/z_m]$,
where 
$\widetilde{D}_m=\C\langle z_1, \dots, z_m,
\wt{\pa_1},\dots, \wt{\pa_m}\rangle$,
$\wt{\pa_k}=\frac{\pa}{\pa z_k}$.
Then $\f$ induces a homomorphism 
$\f^*:\C[x_1^{\pm},\dots,x_m^{\pm}] \to \C[z_1^{\pm},\dots,z_m^{\pm}]$ and 
$\f^*:D_m^* \to \widetilde{D}_m^*$. 
The ideal of $\widetilde{D}_m^*$ generated
by $\f^*(\ell_i)$ ($i=1,\dots,m$) is denoted by $\widetilde{\mathfrak I}$.
Under this change of variables, the rank of the system 
is preserved and the singularities with respect to $z$ are
pull back of those with respect to $x$.
The pull back of the defining polynomial $R(x)$ is the product of 
$p^m$ linear forms:
\begin{equation}
\label{first definitio of singular locus}
R(z)=\prod_{(i_1,\dots,i_m)\in (\Z_p)^m}
(1-\z_p^{i_1}z_1-\cdots -\z_p^{i_m}z_m)
\end{equation}
in $z_1,\dots,z_m$, where $\z_p=\exp(2\pi\sqrt{-1}/p)$ and 
$\Z_p=\Z/(p\Z)=\{1,\dots,p\}$.
Note that the right hand side of
\eqref{first definitio of singular locus}
is invariant under the action of 
$
(z_i)_i\mapsto (\zeta_p^{i_i}z_i)_i
$
for $(i_1, \dots, i_m)\in \Z_p^m$, and it can be written as
a polynomial of $x_1, \dots, x_m$.
Let $\widetilde{\xi_i}$ be the symbol of 
the differential operator 
$\wt\pa _i$ and 
$\init(\widetilde{\mathfrak I})$ be the ideal  
of $\C[z_i^{\pm},\widetilde{\xi}_i]$
consisting of the symbols of $\widetilde{\mathfrak I}$. 
Then the rank and the singular locus are computed from the ideal $\init(\widetilde{\mathfrak I})$.
We consider 
the ideal $\widetilde{\mathcal I}^*$ of
$\C[z_i^{\pm},\xi_i]$ generated by the symbols
$\widetilde{L_i}$ of $\wt\ell_i=\f^*(\ell_i)$, 
Then we have $\widetilde{\mathcal I}^* \subset 
\init(\widetilde{\mathfrak I})$.
To study
the rank and the singular locus of the system of
differential equations,
we show that 
$A[\xi_1,\dots,\xi_m]/\widetilde{\mathcal I}^*$
is a locally free 
$A$-module of rank $p^m$,
where 
$
A=\C[z_i^{\pm},1/R(z)].
$
To see this property, we change generators 
$\widetilde{L_1},\dots, \widetilde{L_m}$
of the ideal 
$\widetilde{\mathcal I}^*$
 by
$\widetilde{M_1},\dots, \widetilde{M_m}$ and show that
$\widetilde{M_1},\dots, \widetilde{M_m}$ form a regular sequence in the ring $A[\widetilde{\xi_i}]_i$.
By computing the Hilbert polynomial with respect to this 
regular sequence, we conclude that
$A[\xi_1,\dots,\xi_m]/(\widetilde{L_1},\dots, \widetilde{L_m})$
is a locally free module of rank $p^m$. 
As is pointed out in the paper \cite{HT}, the ideal 
$\widetilde{\mathcal I}^*$ can be different from
$\init(\widetilde{\mathfrak I})$.
To show that the singular locus
is given by $R(x)=0$, we use $p^m$ independent solutions.

In the forthcoming paper \cite{KMOT}, 
we study the fundamental group of the complement $X$ of 
the singular locus $S(x)$ of the system $\cF_{C}^{p,m}(a,B)$ and  
the monodromy representation of $\cF_{C}^{p,m}(a,B)$.
We show that the fundamental group is generated by $m+1$ loops 
$\rho_0,\rho_1,\dots,\rho_m$  in $X$, and that they satisfy relations 
$$\rho_i\rho_j=\rho_j\rho_i\ (1\le i<j\le m),\quad 
(\rho_0\rho_k)^p=(\rho_k\rho_0)^p\ (1\le k\le m),$$
as elements of the fundamental group. 
We give circuit matrices along $\rho_k$ $(0\le k\le m)$ with respect to 
a fundamental system of solutions to $\cF_C^{p,m}(a,B)$ around 
a point near to the origin. 

After finishing our work, we noticed that in the paper
\cite{MO}, they also treat the same systems of differential equations,
or even more general ones, and obtain the singular loci and 
the ranks of them. 
They compute the common zeros of the ideal 
$\widetilde{\mathcal I}^*$ and continue the argument
using GKZ machinery, which yields the required results
on singularities and ranks.
On the other hand, we directly prove that
$A[\xi_1,\dots,\xi_m]/\widetilde{\mathcal I}^*$
is a locally free module of rank $p^m$
on the set $U=\{R(x)\neq 0\}$
by regular sequence argument.
This method gives an information on the structure of solutions
over $U$.
We hope that our method has an advantage for the study of
the Pfaffian form for this system. 
 
\section{Hypergeometric series $F_C^{p,m}(a,B;x)$}
\label{sec:series}
Let $m$ and $p$ be natural numbers with $m\ge 1$ and $p\ge 2$. 
We put 
$$a=\begin{pmatrix}
a_1\\\vdots\\ a_p
\end{pmatrix},\quad 
B=(b_1,\dots,b_{m})=
\begin{pmatrix}
b_{1,1} &  & b_{1,m}\\
\vdots & \cdots & \vdots\\
b_{p-1,1} &  & b_{p-1,m}\\
1 & & 1
\end{pmatrix}
$$
for complex parameters $a_i$ ($1\le i\le p$), and  
$b_{j,k}$ ($1\le j\le p-1$, $1\le k\le m$) 
with $b_{j,k}\notin -\N=\{0,-1,-2,\dots\}$.
We assume that the last entry $b_{p,k}$ of every column vector 
$b_k$ is $1$.

\begin{definition}
\label{def:HGS}
We define a hypergeometric series in $m$ variables $x_1,\dots,x_m$ 
with parameters $a$ and $B$ as 
\begin{eqnarray}
\label{eq:HGS}
& &F_C^{p,m}(a,B;x)=F_C^{p,m}\left(
\begin{pmatrix} a_1 \\
\vdots \\
a_p
\end{pmatrix},
\begin{pmatrix}
b_{1,1} &  & b_{1,m}\\
\vdots & \cdots & \vdots\\
b_{p,1} &  & b_{p,m}\\
\end{pmatrix};
\begin{pmatrix} x_1 \\
\vdots \\
x_m
\end{pmatrix}
\right)\\ \nonumber
&=&
\sum_{(n_1,\dots,n_m)\in \N^m}
\frac{(a_1,n_1+\cdots +n_m)\cdots (a_p,n_1+\cdots +n_m)}
{\prod_{k=1}^{m}\{(b_{1,k},n_k)\cdots (b_{p-1,k},n_k)n_k!\}}
x_1^{n_1}\cdots x_m^{n_m}.
\end{eqnarray}
\end{definition}
Note that $F_C^{p,m}(a,B;x)$ reduces to 
the generalized hypergeometric series  
$\HGF{p}{p-1}{a_1,\dots,a_p}{b_{1,k},\dots,b_{p-1,k}}{x_k}$ 
when it is restricted to $x_1=\cdots=x_{k-1}=x_{k+1}=\cdots=x_m=0$, and that 
it is Lauricella's hypergeometric series 
$F_C(a_1,a_2,b_{1,1},\dots,b_{m,1};x_1,\dots,x_m)$  
in $m$-variables $x_1,\dots,x_m$ when $p=2$. 
The series $F_C^{p,m}(a,B;x)$ can be regarded as one of multi-variable 
models of the generalized hypergeometric series $_{p}F_{p-1}$ just like 
Lauricella's hypergeometric series $F_C$
is that of the hypergeometric series $_{2}F_{1}$.

\begin{proposition}
\label{prop:converge}
This series absolutely converges in the domain 
$$\D=\{x\in \C^m\mid\sqrt[p]{|x_1|}+\cdots+\sqrt[p]{|x_m|}<1\}.$$
\end{proposition}
\begin{proof}

We modify the proof of \cite[Theorem 8.1]{Kim}.  
By the formula 
$$\frac{1}{\G(s)} =\lim_{N\to\infty} \frac{(s,N)}{(N-1)!N^s},$$
we have 
$$(a_i,n_i)\sim \frac{(n_i-1)! n_i^{a_i}}{\G(a_i)}=
\frac{n_i! n_i^{a_i-1}}{\G(a_i)}
\quad (\N\ni n_i\to \infty).
$$
Since
\begin{eqnarray*}
& &A_{n}=\frac
{(a_1,n_1+\cdots+n_m)\cdots (a_p,n_1+\cdots+n_m)}
{\prod_{k=1}^m\{(b_{1,k},n_k)\cdots (b_{p-1,k},n_k)n_k!\}}\\
&\sim&
\frac{\prod_{\substack{1\le j\le p-1\\1\le k\le m}} 
\G(b_{j,k})}{\G(a_1)\cdots\G(a_p)}
\cdot
\frac{(n_1\!+\!\cdots\!+\!n_m)^{a_1+\cdots+a_p-p}}
{\prod_{k=1}^m n_k^{b_{1,k}+\cdots+b_{p-1,k}-p+1}}\cdot
\left(\frac{(n_1\!+\!\cdots\!+\!n_m)!}{n_1!\cdots n_m!}
\right)^p,
\end{eqnarray*}
as $n_1,\dots, n_m\to \infty$, and 
$$\prod_{k=1}^m n_k^{p-1-b_{1,k}-\cdots-b_{p-1,k}}
\le n_1^{\b}\cdots n_m^{\b}\le 
\frac{(n_1+\cdots +n_m)^{m\b}}{m^{m\b}},$$
for a positive constant $\b\ge \max\limits_{1\le k\le m}
(\re(p-1-b_{1,k}-\cdots-b_{p-1,k}))$, 
there exists a positive constant $C$ such that 
\begin{eqnarray*}
& &|A_{n}x_1^{n_1}\cdots x_m^{n_m}|\\
&\le& C\cdot
(n_1\!+\! \cdots\! +\! n_m)^{\a}\cdot
\left(\frac{(n_1\!+\!\cdots \!+\! n_m)!}{n_1!\cdots n_m!}
\right)^p\cdot|x_1|^{n_1}\cdots |x_m|^{n_m},
\end{eqnarray*}
where $\a=\re(a_1+\cdots+a_p+m\b-p)$.
Thus we have
\begin{eqnarray*}
& &
\sum_{n\in \N^m}
|A_{n}x_1^{n_1}\cdots x_m^{n_m}|
=\sum_{N=0}^\infty\sum_{n_1+\cdots+n_m=N}|A_{n}x_1^{n_1}\cdots x_m^{n_m}|\\
&\le&
C\sum_{N=0}^\infty N^{\a}\sum_{n_1+\cdots +n_m=N}
\left(\frac{(n_1+\cdots+n_m)!}{n_1!\cdots n_m!}
\right)^p|x_1|^{n_1}\cdots |x_m|^{n_m}\\
&\le&
C\sum_{N=0}^\infty N^{\a}
\left(\sum_{n_1+\cdots +n_m=N}
\frac{(n_1+\cdots+n_m)!}{n_1!\cdots n_m!}
|x_1|^{\frac{n_1}{p}}\cdots |x_m|^{\frac{n_m}{p}}\right)^p\\
&=&
C\sum_{N=0}^\infty 
N^{\a}(\sqrt[p]{|x_1|}+\cdots+\sqrt[p]{|x_m|})^{pN}.
\end{eqnarray*}
Hence 
$F_C^{p,m}(a,B;x)$ absolutely converges if $x$ satisfies the 
inequality $\sqrt[p]{|x_1|}+\cdots+\sqrt[p]{|x_m|}<1$. 
\end{proof}

\begin{remark}
\label{rem:diverge}
We can show that 
$\{A_{n}x_1^{n_1}\cdots x_m^{n_m}\mid n_1,\dots,n_m\in \N\}$
is not bounded for $\sqrt[p]{|x_1|}+\cdots+\sqrt[p]{|x_m|}>1$. 
Hence $F_C^{p,m}(a,B;x)$ diverges if $x$ belongs to the exterior of $\D$.
\end{remark}

\section{An Euler type  Integral of $F_C^{p,m}(a,B;x)$ }

\begin{theorem}
\label{th:int-rep}
Suppose that 
$$
a_j,\  
 b_{j,1},\dots,b_{j,m},\ a_j- b_{j,1}-\cdots-b_{j,m}\notin \Z\quad 
(1\le j\le p-1).
$$
The hypergeometric series $F_C^{p,m}(a,B;x)$ admits 
the integral representation of Euler type:
$$
c_\G \!\int_{\De_1\times\cdots\times\De_{p-1}}\!\!
\Big(\prod_{\substack{1\le j\le p-1\\1\le k\le m}} 
t_{j,k}^{-b_{j,k}}\Big)\cdot 
Q(t,x)^{-a_p}\cdot 
\prod_{j=1}^{p-1} L_j(t)^{b_{j,1}+\cdots+b_{j,m}-a_j-m}
\cdot dt, 
$$
where $\De_j$ $(1\le j\le p-1)$ are regularized twisted cycles 
associated to the $m$-simplex
$$\{(t_{j,1},\dots,t_{j,m})\in \R^m\mid 
0\le t_{j,1},\dots,t_{j,m},\ t_{j,1}+\cdots+t_{j,m}\le 1\}$$ 
with exponents $1-b_{j,1},\dots,1-b_{j,m}$, 
$1+b_{j,1}+\cdots+b_{j,m}-a_j-m$ as in \cite[\S4]{G1} for the case $p=2$ and  
in \cite[Theorem 3.1]{KMO1} for $(p,m)=(3,2)$, 
\begin{eqnarray*}
Q(t,x)&=&
1-\frac{x_1}{t_{1,1}\cdots t_{p-1,1}}-\cdots 
-\frac{x_m}{t_{1,m}\cdots t_{p-1,m}},\\
L_j(t)&=&1-t_{j,1}-\cdots-t_{j,m},\\
dt&=&\bigwedge_{j=1}^{p-1} dt_j=\bigwedge_{j=1}^{p-1} 
(dt_{j,1}\wedge \cdots \wedge dt_{j,m}),
\end{eqnarray*}
and the gamma factor $c_\G$ is 
$$
\prod_{j=1}^{p-1}\frac{\G(1-a_j)}
{\G(1-b_{j,1})\cdots \G(1-b_{j,m})
\G(1+b_{j,1}+\cdots+b_{j,m}-a_j-m)}.
$$
\end{theorem}

\begin{proof}
Put $T_k=t_{1,k}\cdots t_{p-1,k}$ for $k=1,\dots,m$, 
and expand the factor $Q(t,x)^{-a_p}$. 
Then we have 
\begin{eqnarray*}
& &\Big(1-\frac{x_1}{T_1}
-\cdots-\frac{x_m}{T_{m}}\Big)^{-a_p}
=\sum_{N=0}^{\infty} \frac{(a_p,N)}{N!}
\Big(\frac{x_1}{T_1}
+\cdots+\frac{x_m}{T_m}\Big)^N\\
&=&\sum_{N=0}^{\infty} \frac{(a_p,N)}{N!}
\sum_{n_1+\cdots+n_m=N}\frac{N!}{n_1!\cdots n_m!}
\Big(\frac{x_1}{T_1}\Big)^{n_1}
\cdots \Big(\frac{x_m}{T_m}\Big)^{n_m}\\
&=&\sum_{n_1,\dots,n_m=0}^\infty \frac{(a_p,n_1+\cdots+n_m)}{n_1!\cdots n_m!}
(T_1^{-n_1}\cdots T_m^{-n_m})\cdot(x_1^{n_1}\cdots x_m^{n_m}).
\end{eqnarray*}
Change the order of the summation and the integration.
The coefficient of $x_1^{n_1}\cdots x_m^{n_m}$ in the integration 
(without the gamma factor $c_\G$) is the product of 
\begin{eqnarray*}
& &\int_{\De_j}t_{j,1}^{-b_{j,1}-n_1}\cdots 
t_{j,m}^{-b_{j,m}-n_m}(1\!-\! t_{j,1}\!-\!\cdots\! -\! t_{j,m})^{b_{j,1}+\cdots
+b_{j+m}-a_j-m}
dt_j\\
&=&
\frac{\G(1\!-\! b_{j,1}\!-\! n_1)\cdots \G(1\!-\! b_{j,m}\!-\!n_m)}
{\G(1-a_j-n_1-\cdots-n_m)}
\G(1\!+\!b_{j,1}\!+\!\cdots\!+\!b_{j,m}\!-\!a_j\!-\!m)
\end{eqnarray*}
for $1\le j\le p-1$. 
Here note that we use 
the regularized twisted cycle $\De_j$ to make the integral converge, 
since the real parts of $1-b_{j,1}-n_1$,\dots, $1-b_{j,m}-n_m$ 
are negative for sufficiently large $n_1,\dots,n_m$.
By using the formulas 
$(\a,s)\G(\a)=\G(\a+s)$, and  
$\G(\a)\G(1-\a)=\dfrac{\pi}{\sin(\pi\a)}$, 
we transform $\G(1-b_{j,k}-n_k)$ into
\begin{eqnarray*}
& &\G(1-b_{j,k}-n_k)=\frac{\pi}{ \sin(\pi(b_{j,k}+n_k))\G(b_{j,k}+n_k)}\\
&=&\frac{\pi}{(-1)^{n_k}\sin(\pi b_{j,k})\G(b_{j,k})(b_{j,k},n_k)}
=\frac{\G(1-b_{j,k})}{(-1)^{n_k}(b_{j,k},n_k)}.
\end{eqnarray*}
Similarly we have
$$\G(1-a_j-n_1-\cdots -n_m)=
\frac{\G(1-a_j)}{(-1)^{n_1+\cdots +n_m}(a_j,n_1+\cdots+n_m)}.
$$
Hence the following gamma factors including $n_1,\dots,n_m$ 
can be expressed in terms of Pochhammer's symbols as
\begin{eqnarray*}
& &\frac{\G(1\!-\! b_{j,1}\!-\! n_1)\cdots \G(1\!-\! b_{j,m}\!-\!n_m)
}
{\G(1-a_j-n_1-\cdots-n_m)}\\
&=&
\frac{\G(1\!-\! b_{j,1})\cdots \G(1\!-\! b_{j,m})
}
{\G(1-a_j)}\cdot 
\frac{(a_j,n_1+\cdots+n_m)}{(b_{j,1},n_1)\cdots(b_{j,m},n_m)}, 
\end{eqnarray*}
which yield the coefficient of $x_1^{n_1}\cdots x_m^{n_m}$ in 
$F_C^{p,m}(a,B;x)$.
\end{proof}

\section{A  system of  hypergeometric differential equations}

We can easily show the following lemma by using actions of Euler's operators 
$\theta_k=x \pa_k$ $(k=1,\dots,m)$ on power functions $x_j^{\a_j}$ of $x_j$ 
$(j=1,\dots,m)$: 
\begin{equation}
\label{eq:theta-rel}
\theta_k (x_j^{\a_j})=\delta_{j,k}\a_jx_j^{\a_j},
\end{equation}
where  $\pa_k$ denotes the holomorphic partial differential operator 
$\frac{\pa}{\pa x_k}$ and $\delta_{j,k}$ denotes Kronecker's symbol. 

\begin{lemma}
\label{lem:HGDE}
The series $F_C^{p,m}(a,B;x)$ satisfies hypergeometric 
differential equations
\begin{align}
\label{differential equations first time}
 &\theta_k(b_{1,k}-1+\theta_k)\cdots (b_{p-1,k}-1+\theta_k)F_C^{p,m}(a,B;x)
\\
\nonumber
=&x_k(a_1+\theta_1+\cdots+\theta_m)\cdots
(a_p+\theta_1+\cdots+\theta_m)F_C^{p,m}(a,B;x)
%
\end{align}
for $k=1,\dots,m$.
\end{lemma}

\begin{definition}
\label{def:System-HGDE} 
We define $\cF_C^{p,m}(a,B)$ as the system of the $m$ 
hypergeometric differential equations in Lemma \ref{lem:HGDE}. 
\end{definition}
We set $\Z_p=\{1,2,\dots,p\}=\Z/(p\Z)$. The element $p\in \Z_p$ is 
often regarded as $0$. 
From now on we assume the following conditions:
\begin{equation}
\label{eq:non-integral}
a_i-\sum_{k=1}^m b_{j_k,k},\notin \Z,
\end{equation}
\begin{equation}
\label{eq:non-integral-addition}
b_{j,k}-b_{j',k}\notin \Z,
\end{equation}
where $1\le i\le p$, $1\le j<j'\le p$, $1\le k\le m$, and  
$J=(j_1,\dots,j_m)\in (\Z_p)^m$.
There are $p^{m+1}$  conditions for \eqref{eq:non-integral}, 
and $mp(p-1)/2$ conditions for \eqref{eq:non-integral-addition}. 

The following proposition gives explicit expressions of independent
solutions to the system of differential equations
\eqref{differential equations first time}. 
To give the statement of the proposition, we introduce
maps $\eta_{j_k}$ ($1\le k\le m$, $j_k\in \Z_p$) 
on the space of exponent parameters $b_k$: 
$$
\eta_{j_k}(b_k)=b_k+(1-b_{j_k,k})(\one_p+e_{j_k}-e_p)
,
$$
where $\one_p=\tr(1,\dots,1)\in \Z^p$ and
$e_j$ is the $j$-th unit column vector.

\begin{remark} If $j_k= p$ then $1-b_{j_k,k}=0$ and 
$\eta_{j_k}=\eta_p$ is the identity map $id_p$, 
otherwise it is a reflection 
$$
\eta_{j_k}=id_p-\frac{2}{\tr v_{j_k} W v_{j_k}} v_{j_k}\tr v_{j_k} W, 
$$
where $v_{j_k}=\one_p+e_{j_k}-e_p=\tr(1,\dots,1,\overset{j_k\textrm{-th}}{2},
1,\dots,1,0)
$ and 
$$
W=p\cdot id_p- \one_p\tr \one_p=\begin{pmatrix}
p-1 & -1 & \cdots & -1\\
-1 & p-1 & \cdots & -1\\
\vdots &\vdots &\ddots &\vdots \\
-1 & -1 &\cdots & p-1
\end{pmatrix},
$$
which is a positive semi-definite matrix with 
eigenvalues $0,\overbrace{p,\dots,p}^{p-1}$.
\end{remark}

\begin{proposition}
\label{prop:fund-solutions}

Under the conditions \eqref{eq:non-integral} and 
\eqref{eq:non-integral-addition}, 
there are linearly independent $p^m$ solutions 
\begin{equation}
\label{eq:fund-solutions}
\Phi_J(a,B;x)
=\Big(\prod_{k=1}^m x_k^{\mu_{J,k}}\Big)
F_C^{p,m}\Big(a+\Sigma_J\one_p,\eta_{j_1}(b_1),\dots,\eta_{j_m}(b_m);x
\Big)
\end{equation}
to $\cF_C^{p,m}(a,B)$ 
on a small neighborhood in $\D$ 
indexed by $J=\{j_1,\dots,j_m\}\in{(\Z_p)}^m$, 
where
$$
\mu_{J,k}=1-b_{j_k,k},\quad 
\Sigma_J=\sum_{k=1}^m\mu_{J,k}=\sum_{k=1}^m(1-b_{j_k,k}).
$$
\end{proposition}

\begin{remark}
If $j_k=p(=0)$ then $\mu_{J,k}=0$ and $\eta_{j_k}(b_k)=b_k$. Thus we have 
$$\Phi_J(a,B;x)=F_C^{p,m}(a,B;x)$$
for $J=\{p,\dots,p\}(=\{0,\dots,0\})$. 
\end{remark}

\begin{proof}
We act the operators 
$\theta_k(b_{1,k}-1+\theta_k)\cdots (b_{p-1,k}-1+\theta_k)$ and 
$x_k(a_1+\theta_1+\cdots+\theta_m)\cdots(a_p+\theta_1+\cdots+\theta_m)$ on 
$\Phi_J(a,B;x)$.
Move the factor $\Big(\prod\limits_{k=1}^m x^{\mu_{J,k}}\Big)$ in $\Phi_J(a,B;x)$ 
to the left of them by using 
$$
\theta_k\cdot (x_j^{\a_j}\cdot f(x))=x_j^{\a_j}\cdot 
(\delta_{j,k}\a_j+\theta_k)\cdot f(x).
$$
Then we have the system 
of differential equations annihilating the series 
$F_C^{p,m}\Big(a+\Sigma_J\one_p,\eta_{j_1}(b_1),\dots,\eta_{j_m}(b_m);x\Big)$. 
Thus  $\Phi_J(a,B;x)$ is a solution to $\cF_C^{p,m}(a,B)$.

The linear independence of the $p^m$ solutions 
$\Phi_J(a,B;x)$ ($J\in (\Z_p)^m$) 
follows from that of the $p^m$ factors $\Big(\prod\limits_{k=1}^m x_k^{\mu_{J,k}}\Big)$ 
under \eqref{eq:non-integral-addition}.
\end{proof}

By Proposition \ref{prop:fund-solutions}, the rank of $\cF_C^{p,m}(a,B)$
is greater than or equal to $p^m$. 
In the next section, we show that it is equal to $p^m$.

\section{Singular locus and the rank of the associated  $D_m$ module}

We study the rank and the singular locus of 
the hypergeometric system $\cF_C^{p,m}(a,B)$. 
Main results in this section are as follows. 
\begin{theorem}
\label{th:rank}
Under the conditions  
\eqref{eq:non-integral} and \eqref{eq:non-integral-addition}.
The system $\cF_C^{p,m}(a,B)$ is of rank $p^m$.
As for the definition 
and properties of the rank of a system of differential equations, 
see Fact \ref{fact:rank}.
\end{theorem}

\begin{theorem}
\label{th:sing-loc}
Let $R(x)=R(x_1,\dots,x_m)$ be a polynomial in $x_1,\dots,x_m$ of degree 
$p^{m-1}$ given by 
\begin{equation}
\label{eq:R(x)}
R(x_1,\dots,x_m)=\prod_{(i_1,\dots,i_m)\in (\Z_p)^m}
(1-\z_p^{i_1}\sqrt[p]{x_1}-\cdots -\z_p^{i_m}\sqrt[p]{x_m})
\end{equation}
where $\z_p=\exp(2\pi\sqrt{-1}/p)$ is the primitive $p$-th root 
of unity.

Under the conditions  
\eqref{eq:non-integral} and \eqref{eq:non-integral-addition},
the singular locus of the 
system $\cF_C^{p,m}(a,B)$ is 
\begin{equation}
\label{eq:sing-locus}
S=S(x)=\{x\in \C^m\mid x_1\cdots x_m R(x)=0\},
\end{equation}
%
For the definition of the singular locus
of a system of differential equations see Fact \ref{fact:sing-loc}.
\end{theorem}

\begin{remark}
We give examples of $R(x)$ for small $m$ and $p$: 
$$
\begin{array}{cc|c}
m & p & R(x) \\
\hline 
2 & 2 &(1-x_1-x_2)^2-4x_1x_2\\[1mm]
2 & 3 &(1-x_1-x_2)^3-27x_1x_2\\[1mm]
2 & 4 &\{(1-x_1-x_2)^2-4x_1x_2\}^2-128x_1x_2(1+x_1+x_2)\\[1mm]
3 & 2 &
\begin{matrix}
\{(1\!+\! x_1\!+\! x_2\!+\! x_3)^2-
4(x_1x_2\!+\! x_1x_3\!+\! x_2x_3\!+\! x_1\!+\! x_2\!+\! x_3)\}^2\\
-64x_1x_2x_3
\end{matrix}
\\
\end{array}
$$
When $m=3$, $p=2$, $R(x)=0$ is called Steiner's surface, refer to \cite{Mu}.
\end{remark}

\subsection{Weyl algebras and systems of differential equations}
In this section, we recall basic facts on systems of
differential equations and the Weyl algebra
$D_m=\C\la x_1,\dots,x_m,\pa_1,\dots,\pa_m\ra$,    
which is a non-commutative $\C$-algebra generated by 
$x_1,\dots,x_m$ and 
$\pa_1,\dots,\pa_m$ with relations  
\begin{equation}
\label{eq:rel-Weyl}
x_ix_j=x_jx_i,\quad \pa_i\pa_j=\pa_j\pa_i,\quad 
\pa_i x_j=x_j \pa_i+\delta_{i,j},
\end{equation}
for $1\le i,j\le m$. 
By the relation \eqref{eq:rel-Weyl}, any element  $\Theta \in D_m$ can be
expressed as a finite sum 
\begin{equation}
\label{eq:element-WA}
\Theta =\sum_{\a,\b\in \N^{m}} c_{\a,\b}x_1^{\a_1}\cdots x_m^{\a_m}
\pa_1^{\b_1}\cdots \pa_m^{\b_m},
\end{equation}
where $c_{\a,\b}\in \C$ vanish for sufficiently large 
$|\a|=\a_1+\cdots +\a_m$ and $|\b|=\b_1+\cdots +\b_m$ 
for $\a=(\a_1,\dots,\a_m)$, $\b=(\b_1,\dots,\b_m)$.
We set 
$$
\ord(\Theta)=\max\{|\b|\mid c_{\a,\b}\ne 0\}.
$$
The subspaces 
$$
D_m^{\leq d}=\{\Theta\in D_m\mid\ord(\Theta)\leq d\}\quad (d\in \N)
$$
form a filtration of $D_m$ compatible with multiplication, 
and we have the associated graded ring
$$
\gr(D_m)=\bigoplus_{d\in \N} \gr^d(D_m),\quad 
\gr^d(D_m)=D_m^{\le d}/D_m^{\le d-1}, 
$$
of $D_m$. 
Note that $\gr(D_m)$ is a commutative ring
isomorphic to the polynomial ring 
$\C[x,\xi]=\C[x_1,\dots,x_m,\xi_1,\dots,\xi_m]$
in $2m$ variables 
$x_1,\dots,x_m,$ $\xi_1,\dots,\xi_m$, where $\xi_i$ is the image of $\pa_i$
under the natural projection.
For $\Theta\in D_m^{\leq d}$ in \eqref{eq:element-WA},
its image in $\gr^d(D_m)$ is equal to
$$
\init^d(\Theta)=\sum_{\substack{\a,\b\in \N^{m}\\ \b_1+\cdots+\b_m=d}}
c_{\a,\b}x_1^{\a_1}\cdots x_m^{\a_m}
\xi_1^{\b_1}\cdots \xi_m^{\b_m},
$$
which is called the $d$-initial form of $\Theta$.

Let $\ell_1, \dots, \ell_r$ be elements in $D_m$, 
and $\mathfrak{I}$ be the left ideal of $D_m$ generated by them.
By regarding the symbols $\pa_1,\dots,\pa_m$ as the 
holomorphic partial differential 
operators, we have a system of linear differential equations 
\begin{equation}
\label{general differential equation}
\ell_1\cdot f(x)=0,\dots, \ell_r\cdot f(x)=0,
\end{equation}
with respect to unknown function $f(x)=f(x_1,\dots,x_m)$. 
Note that this system is equivalent to 
$$\ell \cdot f(x)=0 
$$
for any $\ell\in \mathfrak{I}$.
In this way, systems of linear differential equations correspond 
one to one to left ideals of $D_m$. 
By this correspondence, we identify a left ideal of
$D_m$ with a system of linear differential equations. 
For the left ideal $\mathfrak{I}$ of $D_m$, its 
initial ideal $\init(\mathfrak{I})$ in $\C[x,\xi]$ 
is defined by the ideal 
$$
\bigoplus_d\{ \init^d(\Theta)\mid \Theta\in \mathfrak{I}\cap D_m^{\leq d}
\}.
$$

\begin{fact}[{\cite[Definition~1.4.8]{SST}}]
\label{fact:rank}
The rank of $\mathfrak I$ is defined by
\[
\rank\,(\mathfrak I)=\dim_{\C(x)} \big(
\C(x)[\xi]\,/\, (\C(x)[\xi]\cdot \init\,(\mathfrak I))\big),
\]
where the quotient space in its right hand side is regarded as 
a vector space over the rational function field $\C(x)$.
It is known that $\rank\,(\mathfrak I)$  coincides with the dimension of 
the space of solutions to  
\eqref{general differential equation}
around a general point.
\end{fact}

The characteristic variety of $\mathfrak{I}$
is defined by 
$$
\ch(\mathfrak I) = \{(x,\xi)\in \C^{2m}\mid L(x,\xi)=0\  \textrm{ for any }
 L\in \init(\mathfrak I)\}.
$$

\begin{fact}[{\cite[Section 1.4]{SST}}]
\label{fact:sing-loc}
The singular locus $\sing\,(\mathfrak I)$ of $\mathfrak{I}$
is defined by the Zariski closure of 
$$
\pr_x (\ch(\mathfrak I)\setminus \{(x,\zero_m)
\in \C^{2m}\mid x\in \C^m\}),
$$
where 
$\pr_x:(x_1,\dots,x_m,\xi_1,\dots,\xi_m)\mapsto (x_1,\dots,x_m).$
The space of solutions form a local system on the complement 
of the singular locus. 
\end{fact}

\subsection{Hypergeometric differential equations and abelian covering}
Now we apply the above method to the study of the system of
hypergeometric differential equations.
We define $m$ elements
\begin{equation}
\label{eq:ell-k}
\ell_k
= \prod_{i=1}^p (x_k\pa_k +b_{ik}-1)- x_k \prod_{i=1}^p 
\left( \sum_{j=1}^m x_j\pa_j + a_i \right)
\end{equation}
in $D_m$ for $k=1,\dots,m$, and  a left ideal  
\begin{equation}
\label{eq:ideal-I}
\mathfrak{I}=(\ell_1,\dots,\ell_m)
\end{equation}
of $D_m$ generated by $\ell_1,\dots,\ell_m$.
By Lemma \ref{lem:HGDE}, the hypergeometric system 
$\cF_C^{p,m}(a,B)$ corresponds to the ideal $\mathfrak{I}$
as in the previous subsection.


We consider a ramified abelian covering 
$$
\varphi:\C^m_z=\C^m\ni (z_1, \dots, z_m)\mapsto 
(x_1, \dots, x_m)=(z_1^p, \dots, z_m^p)\in \C^m_x=\C^m
$$
of $\C^m_x$ to decompose the polynomial $R(x)$ in to a product of
linear functions. That is
\begin{equation}
\label{eq:R(z)}
R(z)=\prod_{(i_1,\dots,i_m)\in (\Z_p)^m}
(1-\z_p^{i_1}z_1-\cdots -\z_p^{i_m}z_m).
\end{equation}
We set
\begin{equation}
\label{eq:sing-locus}
S(z)=\{z\in \C_x^m\mid z_1\cdots z_m R(z)=0\}.
\end{equation}
The covering transformation group of $\f$ is generated by
$$
\sigma_k:(z_1,\dots, z_k,\dots, z_m)\mapsto (z_1, \dots, \zeta_p z_k,\dots, z_m)
\quad (k=1,\dots,m).
$$
Since $\C(z_1,\dots,z_m)/\C(x_1,\dots,x_m)$ is a Galois extension 
with the Galois group $(\Z_p)^m$,  
$R(x)$ is irreducible in $\C[x_1,\dots,x_m]$. 
The hypersurface $R(x)=0$ is unirational, 
since it is covered by a rational variety
parameterized by $(z_1,\dots,z_{m-1})\in \C^{m-1}$
under the morphism
$$
(x_1,\dots,x_{m-1},x_m)=(z_1^p,\dots,z_{m-1}^p,(1-z_1-\cdots-z_{m-1})^p).
$$

To show Theorems \ref{th:rank} and \ref{th:sing-loc}, 
we introduce localized Weyl algebras $D_m^*$ and $\wt{D}_m^*$ 
on $x_1, \dots, x_m$ and $z_1, \dots, z_m$,
and a homomorphism 
$\f_*:D_m^*\to \wt{D}_m^*$ 
of non-commutative $\C$-algebras.
The localized Weyl algebra 
$D_m^*=\C\la x_1^{\pm},\dots,x_m^{\pm},\pa_1,\dots,\pa_m\ra$
is defined as a non-commutative $\C$-algebra generated
by $x_1^\pm,\dots,x_m^\pm $ and 
$\pa_1,\dots,\pa_m$ with relations extending  \eqref{eq:rel-Weyl}.
We similarly define the Weyl algebra 
$\wt{D}_m=\C\la z_1,\dots,z_m,\wt{\pa}_1,\dots,\wt{\pa}_m\ra$ 
and the localized Weyl algebra 
$\wt{D}_m^*=\C\la z_1^\pm,\dots,z_m^\pm,\wt{\pa}_1,\dots,\wt{\pa}_m\ra$
on $z_1,\dots,z_m$. 
Since the map $\varphi^*$ satisfies 
$$
\f^*(x_k)=z_k^p,\quad  
\f^*\circ \pa_k =\frac{1}{pz^k }\wt{\pa_k}\circ \f^*,
$$
and preserves the relations defining $D_m^*$ and $\wt{D}_m^*$, 
it gives rise to a homomorphism of non-commutative algebras
from ${D}_m^*$ to $\wt{D}_m^*$.

Let $\mathfrak{I}^*$ be the left ideal in $D_m^*$ generated by 
$\ell_1,\dots,\ell_m$ given in \eqref{eq:ell-k}.  
Let $\wt{\ell}_k$ denote the images of $\ell_k\in D_m^*$ under the map $\f^*$.
Then we have 
\begin{equation}
\label{eq:f-ell-k}
\wt{\ell}_k=\prod_{i=1}^p (\frac{1}{p}z_k\wt{\pa}_k +b_{ik}-1)- 
z_k^p \prod_{i=1}^p \left( \frac{1}{p}\sum_{j=1}^m z_j\wt{\pa}_j + a_i \right) 
\end{equation}
for $k=1,\dots,m$. 
Note that $\wt{\ell}_k\in \wt{D}_m^*$ belongs to the Weyl algebra 
$\wt{D}_m$.  
We set 
\begin{equation}
\label{eq:SDE in Dm}
\wt{\mathfrak{I}}=(\wt{\ell}_1,\dots, \wt{\ell}_m),
\end{equation} 
which is  the left ideal of $\wt{D}_m$ 
generated by $\wt{\ell}_1,\dots, \wt{\ell}_m$.

To study the rank and the singular locus of the systems of differential
equations associated to $\mathfrak I$ and $\wt{\mathfrak I}$,
we consider the relation between the homomorphism $\f^*$ and
the pull back of a holomorphic function on an open set of $\C^m_x$.
Let $U_x$ be a contractible open set in 
$(\C_x^{\times})^m=\{x\in \C_x^m\mid x_1\cdots x_m\ne0\}$ 
and let $U_z$ be a connected component of $\f^{-1}(U_x)$.
Then the restriction $\f_U$ of $\f$ to $U_z$ is biholomorphic.
Let  $f(x)$ be a holomorphic function on $U_x$ and 
$\varphi_U^*(f(x))$
be its pull back by $\varphi_U$. Then for an element $\ell\in D_m^*$, 
we have
\begin{equation}
\label{pull back of hol fun and diff op}
\varphi^*(\ell)\cdot \varphi_U^*(f(x))=\varphi_U^*(\ell\cdot f(x)).
\end{equation}

As a consequence, we have the following proposition.
\begin{proposition}
\label{same rank by covering}
Let $z$ be any point in $(\C_z^{\times})^m=\{z\in \C_z^m\mid z_1\cdots z_m\ne 0\}$. 
Then we have an isomorphism between
the space of solutions to 
$\mathfrak I$ around $\f(z)$ and that to $\wt{\mathfrak I}$ 
around $z$.
As a consequence, the ranks of $\mathfrak I$ and $\wt{\mathfrak I}$ 
are the same. 
\end{proposition}


We consider the associate graded ring of $\wt{D}_m^*$ defined by
$$
\gr(\wt{D}_m^*)=\C[z^\pm,\frac{1}{R},\wt{\xi}]=
\C[z_1^\pm,\dots,z_m^\pm,\frac{1}{R(z)},\wt{\xi}_1,\dots,\wt{\xi}_m],
$$ 
which is the localization of $\gr(\wt{D}_m)=\C[z,\wt{\xi}]$ by $z_1,\dots,z_m$ 
and $R=R(z)$ 
in \eqref{eq:R(z)}.
We set 
\begin{equation}
\label{eq:Lk}
\wt{L}_k = p^p\init^p(\wt{\ell}_k) = (z_k \wt{\xi}_k)^p - z_k^p  
\left( \sum_{j=1}^m z_j \wt{\xi}_j \right)^p.
\end{equation}
The ideal of $\C[z^\pm,\frac{1}{R},\wt\xi]$ generated by 
$\wt{L}_1,\dots,\wt{L}_m$ is denoted by $\wt{\mathcal I}^*=(\wt{L}_1,\dots, \wt{L}_m)$.
Then we have
$$
\wt{\mathcal{I}}^*\subset \init(\wt{\mathfrak{I}})\otimes_{\C[z,\wt\xi]}
\C[z^\pm,\frac{1}{R},\wt\xi].
$$
By setting 
$$
\A=\C[z^\pm,\frac{1}{R}]=\C[z_1^\pm,\dots,z_m^\pm,\frac{1}{R(z)}],
$$ 
we regard $\C[z^\pm,\frac{1}{R},\wt\xi]$ as 
an $\A$-algebra equipped with the grading by the total degree 
with respect to $\wt{\xi}_1, \dots, \wt{\xi}_m$.
Since the ideal $\wt{\mathcal I}^*$ is a homogeneous ideal with respect
to the grading, the quotient ring
$
\C[z^\pm,\frac{1}{R},\wt\xi]/\wt{\mathcal I}^*
$ 
becomes a graded $\A$-algebra. The following theorem is a key to show
Theorems \ref{th:rank} and \ref{th:sing-loc}, which will be proved in the next
subsection.
\begin{theorem}
\label{locally free of exact rank}
The quotient ring 
$\C[z^\pm,\frac{1}{R},\wt\xi]/\wt{\mathcal I}^*$
is a locally free $\A$-module of finite rank $p^m$.
\end{theorem}

\subsection{The structure of the characteristic variety}
Let $\mu_p=\{\zeta_p^j\mid 0\leq j\leq p-1\}$ 
be the set of $p$-th roots of unity.
To prove Theorem \ref{locally free of exact rank}, 
we prepare some propositions and a lemma. 
We set 
\begin{align*}
&M_k 
=\wt{\xi}_k^p-\wt{\xi}_m^p=\prod_{\chi \in \mu_p}
(\wt{\xi}_k-\chi\wt{\xi}_m),\quad (1\leq k\leq m-1),
\\
&M_m  = \wt{\xi}_m^p -  
\left( \sum_{j=1}^m z_j \wt{\xi}_j \right)^p
=\prod_{\chi\in \mu_p}(\wt{\xi}_m-\chi(\sum_{j=1}^m z_j \wt{\xi}_j)).
\end{align*}
Then it is easy to see that
$$
\wt{\mathcal I}=(M_1, \dots,M_{m-1},M_m).
$$
We set $\A_0=\A[\wt{\xi}_m]$ and define an $\A$-algebra $\A_k$ by 
$$
\A_k=\A[\wt{\xi}_m,\wt{\xi}_1, \dots, \wt{\xi}_k]/(M_1,\dots,M_k)
$$
for $k=1, \dots, m-1$. Since $M_1, \dots, M_m$ are homogeneous with respect to 
$\wt{\xi}_1,\dots,\wt{\xi}_m$, the ideal $(M_1,\dots,M_k)$ is a homogeneous ideal of
$\A[\wt{\xi}_m,\wt{\xi}_1, \dots, \wt{\xi}_k]$, 
and the grading on $\A[\wt{\xi}_m,\wt{\xi}_1, \dots, \wt{\xi}_k]$
induces that on $\A_k$.
Since $M_k={\wt\xi_k}^p-{\wt\xi_m}^p$ is a monic polynomial in $\A_{k-1}[\wt{\xi}_{k}]$, 
it is a non-zero divisor.
As a consequence,
the elements $M_1, \dots, M_k$ form a regular sequence of 
$\A[\wt{\xi}_1,\dots, \wt{\xi}_k]$ for $k=1, \dots, m-1$.

We have a natural isomorphism
$\A_{k}\simeq \A_{k-1}[\wt{\xi}_{k}]/(M_{k})$
for $k=1,\dots, m-1$, which is compatible with the grading.

\begin{proposition}
\begin{enumerate}
 \item 
The $A$-algebra $\A_{m-1}$ becomes 
a free $\A_0$-module of rank $p^{m-1}$.
\item
The degree $d$-part $\A_{m-1}^{(d)}$ of $\A_{m-1}$ is a free $\A$-module of 
finite rank $r_{m-1,d}$, where $r_{m-1,d}$ is defined by 
the coefficient of $t^d$ of the power series expansion for 
$\dfrac{(1-t^p)^{m-1}}{(1-t)^m}$, i.e., 
\begin{equation}
\label{gen fn upto m-1} 
\sum_{i=0}^{\infty}r_{m-1,d}t^d=\dfrac{(1-t^p)^{m-1}}{(1-t)^m}. 
\end{equation}

\end{enumerate}
\end{proposition}
\begin{proof}
(1) By using the inductive relation
$$
\A_{k}=\A_{k-1}[\wt{\xi}_{k}]/(M_{k})=\bigoplus_{i=0}^{p-1} \wt{\xi}_k^i\A_{k-1},
$$ 
we have
$$
\A_{m-1}=\bigoplus_{0\leq i_1,\dots, i_{m-1}\leq p-1}
\wt{\xi}_1^{i_1}\cdots \wt{\xi}_{m-1}^{i_{m-1}}\A_0. 
$$

\noindent
(2) The statement is a consequence of (1).
\end{proof}

In the following, we show that $M_m$ is a non-zero divisor in $\A_{m-1}$.
For $k=1,\dots, m-1$, it is easy to see 
that the map 
$$
\begin{matrix}
\phi_k: \A_k=\A_{k-1}[\wt{\xi}_{k}]/(M_{k}) &\xrightarrow{ } &
(\A_{k-1})^{p}
\\
\text{\rotatebox{90}{$\in$}}
& &
\text{\rotatebox{90}{$\in$}}
\\
f(\wt{\xi}_{k})&\mapsto&
(f(\chi\wt{\xi}_m))_{\chi\in\mu_p}
\end{matrix}
$$
is well defined, and preserves the grading.
We have the following. 
\begin{lemma}
\label{injectivity and non zero div}
\begin{enumerate}
 \item 
For $k=1,\dots, m-1$,
$
\phi_k
$
is injective.
\item
The image $\phi_k(\wt{\xi}_m)$ of $\wt{\xi}_{m}\in A_k$ under $\phi_k$ 
is a non-zero divisor in $(A_{k-1})^p$.
\end{enumerate}
\end{lemma}
\begin{proof}
The statements (1) and (2) for $k$ are denoted by ($1_k$)
and ($2_k$), respectively.
We prove that the statements hold by the double induction:
\begin{enumerate}
\renewcommand{\labelenumi}{\alph{enumi}).}
\item
($2_{0}$), 
 \item
($2_{k-1}$)$\Rightarrow$($1_k$), 
\item
``($2_{k-1}$) and ($1_k$)''$\Rightarrow$($2_k$).
\end{enumerate}

\smallskip\noindent
a).  
Since $\A_0$ is the polynomial ring in $\wt{\xi}_m$ over the ring $\A$,
($2_{0}$) obviously holds. 

\smallskip\noindent
b).  Let $F(\wt{\xi}_k)$ be an element in $\A_{k-1}[\wt{\xi}_k]$
and suppose that $F(\zeta_p^i\wt{\xi}_m)=0$ for $i=0, \dots, p-1$.
We inductively prove that $F(\wt{\xi}_k)$ is divisible by 
$h_s(\wt{\xi}_k)=(\wt{\xi}_k-\zeta_p^0\wt{\xi}_m)\cdots (\wt{\xi}_k-\zeta_p^{s-1}\wt{\xi}_m)$.
Suppose that $F(\wt{\xi}_k)=h_s(\wt{\xi}_k)g(\wt{\xi}_k)$ and $F(\zeta_p^s\wt{\xi}_m)=0$.
Since
$$
h(\zeta_p^{s}\wt{\xi}_m)=(\zeta_p^{s}-\zeta_p^0)\cdots (\zeta_p^{s}-\zeta_p^{s-1})\wt{\xi}_m^{s}
$$
is a non-zero divisor in $\A_{k-1}$ by the assumption,
we have $g(\zeta_p^{s}\wt{\xi}_m)=0$.
Therefore $g(\wt{\xi}_k)$ is divisible by $\wt{\xi}_k-\zeta_p^s\wt{\xi}_m$
and $F(\wt{\xi}_k)$ is divisible by $h_{s+1}(\wt{\xi}_k)$.
As a consequence, $F(\wt{\xi}_k)$ is divisible by $M_k=\wt{\xi}_k^p-\wt{\xi}_m^p$, and
we have ($1_k$).

\smallskip\noindent
c). 
The image $\phi_k(\wt{\xi}_m)$ of $\wt{\xi}_m$
under the map $\phi_k$ is equal to $(\wt{\xi}_m,\dots, \wt{\xi}_m)$
in $(\A_{k-1})^p$, which is a non-zero divisor of $(\A_{k-1})^p$ by the assumption
($2_{k-1}$). Since the map $\phi_k$ is injective by the assumption
($1_{k}$), the image of $\wt{\xi}_m$ in $\A_k$ is a non-zero divisor.
\end{proof}
For $1\leq k\leq m-1$ and ${\widehat\chi}=(\chi_1,\dots,\chi_k) \in (\mu_p)^k$, 
we define homomorphisms 
$$
\psi_{k,\widehat\chi}: A_k \ni f(\wt{\xi}_1,\dots,\wt{\xi}_k)\mapsto 
f(\chi_1 \wt{\xi}_m,\dots,\chi_k\wt{\xi}_m)\in A_0,
$$
and set 
\begin{equation}
\label{eval map}
\psi_k:A_k\ni f
\mapsto 
(\psi_{k,\widehat{\chi}}(f) 
)_{\widehat{\chi}\in (\mu_p)^k}\in (A_0)^{p^k}.
\end{equation}

\begin{proposition}
\label{regular seqence for Mm}
\begin{enumerate}
 \item 
The map $\psi_k$ defined in \eqref{eval map} is homogeneous and injective. 
\item
The image of the element $M_m$ in $\A_{m-1}$ is a non-zero divisor, i.e.
the map
$\A_{m-1} \to \A_{m-1}$
defined by the multiplication of $M_m$
is an injective homogeneous $\A$-homomorphism of degree $p$.
As a consequence,
the elements $M_1, \dots, M_m$ form a regular sequence of
 $\A[\wt{\xi}_1,\dots, \wt{\xi}_{m-1}, \wt{\xi}_m]$. 
\end{enumerate}
\end{proposition}
\begin{proof}
(1) 
Since the map $\psi_k$ is equal to the composite of
\begin{equation}
\label{eval map2}
\A_{k} \xrightarrow{\phi_k} \A_{k-1}^p \xrightarrow{\phi_{k-1}} \cdots
\xrightarrow{\phi_2}
 \A_{1}^{p^{k-1}}  \xrightarrow{\phi_1} \A_{0}^{p^{k}}, 
\end{equation}
and the maps appeared in \eqref{eval map2}
are injective by Proposition \ref{injectivity and non zero div}, 
the composite map is injective.

\smallskip\noindent 
(2)
We consider the image of $\psi_{m-1}(M_m)$ of $M_m$  under the injective map
$$
\psi_{m-1}:\A_{m-1}\to (\A_0)^{p^{m-1}}. 
$$
The $\widehat{\chi}=(\chi_1,\dots,\chi_{m-1})$-component of $\psi_{m-1}(M_m)$
under the homomorphism $\psi_{m-1}$ is equal to $C_{\widehat{\chi}}\wt{\xi}_m^p$, where
\begin{align*}
C_{\widehat{\chi}}&=\big(1-(\chi_1z_1+\cdots +\chi_{m-1}z_{m-1}+z_m)^p\big)
\\
&=\prod_{i=0}^{p-1}\big(1-\zeta_p^{i}(\chi_1z_1+\cdots +\chi_{m-1}z_{m-1}+z_m)\big).
\end{align*}
Since $C_{\widehat{\chi}}$ is an invertible element in $\A$, 
the element $\psi_{m-1}(M_m)$ in $\A_0^{p^{m-1}}$ is a non-zero
divisor. Since the map $\psi_{m-1}$ is injective, 
the image of $M_m$ in $\A_{m-1}$ is a non-zero divisor.
\end{proof}

Let $z^{(0)}=(z_1^{(0)},\dots, z_m^{(0)})$ be a point in 
$\C^m_z$ such that 
$$z_1^{(0)}\cdots z_m^{(0)}R(z_1^{(0)},\dots, z_m^{(0)})\neq 0.$$
The residue field defined by the evaluation map $\A \to \C$ at 
the point $z^{(0)}$ is denoted by $\kappa(z^{(0)})(\simeq \C)$.

\begin{proposition}
\label{reg seq specialization}
\begin{enumerate}
\item
The homogeneous elements $M_1, \dots, M_{m}$ form a regular sequence in 
$\kappa(z^{(0)})[\wt{\xi}_1,\dots, \wt{\xi}_{m-1}, \wt{\xi}_m]$. Especially,
the map
$\A_{m-1}\otimes_\A \kappa(z^{(0)}) \to \A_{m-1}\otimes_\A \kappa(z^{(0)})$
defined by the multiplication of $M_m$
is an injective homogeneous $\kappa(z^{(0)})$-homomorphism of degree $p$.
\item
Let $\big[\A_{m-1}\otimes \kappa(z^{(0)})\big]^{(d)}$ be 
the degree $d$-part of $\A_{m-1}\otimes \kappa(z^{(0)})$.
Then its dimension is equal to $r_{m-1,d}$  
in \eqref{gen fn upto m-1}.
\end{enumerate} 
\end{proposition}
\begin{proof}
(1)
We can prove the statement similarly to 
Proposition \ref{regular seqence for Mm}.

\noindent
(2) Since the sequence of homogeneous elements $M_1, \dots, M_{m-1}$ of
degree $p$ is a regular sequence, the generating function of 
$\A_{m-1}\otimes_\A \kappa(z^{(0)})$   is equal to the power series 
in \eqref{gen fn upto m-1}. 
\end{proof}

Before proving Theorem \ref{locally free of exact rank},
we recall the following lemma of ring theory.
\begin{lemma}
\label{injective specialization}
Let $A$ be an integral Noetherian local ring with the maximal ideal $\mathcal M$.
Let $M$ and $N$ be free $A$ modules of finite rank
and $f:M \to N$ be an injective homomorphism. 
Let $\kappa=A/\mathcal M$ be the residue field of $A$.
If the induced map 
$$
f\otimes_A \kappa:M\otimes_A \kappa \to N\otimes_A \kappa
$$
is injective, then the cokernel of $f$ is a free $A$ module.
\end{lemma}
\begin{proof}
The lemma is proved by Nakayama's lemma and the fact that 
a self surjective homomorphism
of a Noetherian module is an isomorphism.
Let $m$ and $n$ be the rank of $M$ and $N$, and  
choose elements
$v_1, \dots, v_{n-m}$ in $N$ whose images 
under the projection form a basis of the vector space 
$(N\otimes \kappa)/(f(M)\otimes \kappa)$.
We consider a homomorphism $g:A^{n-m}\to N$ by setting
$g(a_1, \dots, a_{n-m})=\sum a_iv_i$. By Nakayama's Lemma,
the homomorphism $h:A^n\simeq M\otimes A^{n-m} \xrightarrow{f+g}N
\simeq A^n$ is a self surjective homomorphism of the Noetherian module
$A^n$. The homomorphism $h$ is an isomorphism.
In fact the increasing sequence of submodules $K_n=\{\ker(h^n)\}$
stabilizes for a certain $N$, i.e. $K_N=K_{N+1}$.
Let $a$ be in $\ker(h)$. Then there exists an element $b\in K_{N+1}$
such that $h^{N}(b)=a$. Thus $b$ is in $K_{N+1}$ and 
as a consequence $b$ is in $K_N$.
This implies $a=0$. Hence, $h$ is injective and surjective.
Therefore the cokernel of $f$ is isomorphic to $A^{n-m}$. 
\end{proof}
\begin{proof}
[Proof of Theorem \ref{locally free of exact rank}]
The homomorphism
\begin{equation}
\label{graded global}
\A_{m-1}^{(k)}\to \A_{m-1}^{(k+p)} 
\end{equation}
obtained by the multiplication of $M_m$ is an injective 
homomorphism between  free $\A$ modules of finite rank 
by Proposition \ref{regular seqence for Mm}.
We obtain the homomorphism  
$$
\big[\A_{m-1}\otimes \kappa(z^{(0)})\big]^{(k)}
\to \big[\A_{m-1}\otimes \kappa(z^{(0)})\big]^{(k+p)}
$$
by tensoring with  $\kappa(z^{(0)})$. 
It is also injective for any $z^{(0)} \in Z$ 
by Proposition \ref{reg seq specialization}.
Thus by Lemma \ref{injective specialization}\red{,}
the cokernel of \eqref{graded global} is locally free.

The quotient ring $\C[z^\pm,\frac{1}{R},\wt\xi]/\wt{\mathcal I}^*$
is isomorphic to the cokernel of $\A_{m-1}\to \A_{m-1}$
obtained by the multiplication map $M_m$.
Since its generating function is
$$
\dfrac{(1-t^p)^{m-1}}{(1-t)^m}-t^p
\dfrac{(1-t^p)^{m-1}}{(1-t)^m}=(\sum_{i=0}^{p-1}t^i)^m,
$$
$\C[z^\pm,\frac{1}{R},\wt\xi]/\wt{\mathcal I}^*$
is a locally free $\A$-module
of rank $p^m$.
\end{proof}

\begin{proof}[Proof of Theorem \ref{th:rank}]
It is known that $\rank\,(\mathfrak I)$ coincides with the rank of $\cF_C^{p,m}(a,B)$. 
By Proposition \ref{same rank by covering}, we have 
$\rank\,(\mathfrak I)=\rank\,(\wt{\mathfrak I})$.
Since 
$\wt{L_1}\ldots, \wt{L_m}\in \init\,(\wt{\mathfrak I})$, and
$$
\C(z)[\wt\xi]\cdot \wt{\mathcal I}=\sum_{i=1}^m\C(z)[\wt\xi]\wt{L_i}
\subset \C(z)[\wt\xi]\cdot\init\,(\wt{\mathfrak I})
$$
as ideals in $\C(z)[\xi]$, we have 
\[
\rank\,(\mathfrak I)=\rank\,(\wt{\mathfrak I}) \le \dim_{\C(z)} (\C(z)[\wt\xi]\,/\, \wt{\mathcal I}\cdot \C(z))=p^m.
\]
Since there exist linearly independent $p^m$ solutions to
$\cF_C^{p,m}(a,B)$ under the conditions  
\eqref{eq:non-integral} and \eqref{eq:non-integral-addition}
by Proposition \ref{prop:fund-solutions}, 
we have $p^m\leq \rank\,(\mathfrak I)$. Thus we have
$\rank(\mathfrak I)=p^m$.
\end{proof}


We consider Fact\ \ref{fact:sing-loc} for the 
left ideal $\wt{\mathfrak I}=\la \wt{\ell_1},\dots,\wt{\ell_m}\ra$ to
study the singular locus $\sing(\wt{\mathfrak I})$ in 
$(\C_x^{\times})^m$.
\begin{theorem}
The singular locus $\sing(\wt{\mathfrak I})$ of the left ideal $\wt{\mathfrak I}$ 
is contained in $S(z)$.
\end{theorem}
\begin{proof}
We set 
$$
\ch' = \{(z,\xi)\in (\C_z^m-S(z))\times \C^{m}\mid \wt{L}(z,\wt\xi)=0\  \textrm{ for any }
 \wt{L}\in \wt{\mathcal I}^*\}.
$$  
Since
$$
\wt{\mathcal{I}}^*\subset \init(\wt{\mathfrak{I}})\otimes_{\C[z,\wt\xi]}
\C[z^\pm,\frac{1}{R},\wt\xi],
$$
we have 
$$
\ch'\supset \ch(\wt{\mathfrak I})\cap (\C_z^m-S(z))\times \C^{m}. 
$$
By Proposition \ref{regular seqence for Mm}
and Proposition \ref{reg seq specialization}, we have
\begin{align*}
\emptyset=&\pr_z (\ch'\setminus \{(z,\zero_m)
\in (\C_z^m-S(z))\times \C^{m}\mid z\in Z\})
\\
&\supset \sing(\wt{\mathfrak I})\cap (\C_z^m-S(z)).  
\end{align*}
Therefore we have $\sing(\wt{\mathfrak I})\subset S(z)$.
\end{proof}

\begin{proof}[Proof of Theorem \ref{th:sing-loc}]
Let  $\f^*\cF_C^{p,m}(a,B)$ be the system of differential equation
defined by $\wt{\mathfrak I}$. Since the spaces of solutions to
$\cF_C^{p,m}(a,B)$ and $\f^*\cF_C^{p,m}(a,B)$ have non-trivial 
monodromy representations along $\{x_i=0\}$ and along $\{z_i=0\}$ for $1\leq i\leq m$,
respectively, we have
$$
\{x_1\cdots x_m=0\}\subset \sing(\mathfrak I),\quad
\{z_1\cdots z_m=0\}\subset\sing(\wt{\mathfrak I}).
$$ 
Moreover, since the morphism $(\C_z^{\times})^m \to (\C_x^{\times})^m$
is unramified, the horizontal morphisms of the following  
diagram are isomorphisms:
$$
\begin{matrix}
\sing(\mathfrak I)\cap (\C_z^{\times})^m
&\xrightarrow{\simeq}&\f^{-1}(\sing(\wt{\mathfrak I})\cap (\C_x^{\times})^m)
\\
\cap & & 
\\
S(z)\cap (\C_z^{\times})^m
&\xrightarrow{\simeq}&\f^{-1}(S(x)\cap (\C_x^{\times})^m).
\end{matrix}
$$
Therefore to prove Theorem \ref{th:sing-loc}, it is enough 
to show that the inclusion
\begin{equation}
\label{inclusion is equal}
\sing(\wt{\mathfrak I})\cap (\C_x^{\times})^m\subset
S(x)\cap (\C_x^{\times})^m
\end{equation}
is actually an equality.
Since $S(x)\cap (\C_x^{\times})^m$ is an
irreducible divisor, 
if the inclusion \eqref{inclusion is equal} is a
proper inclusion, the codimension of
$\sing(\wt{\mathfrak I})\cap (\C_x^{\times})^m$
is greater than or equal to $2$.


Therefore, the monodromy representation of
$\pi_1((\C_x^{\times})^m-\sing(\wt{\mathfrak I}),\dot{x})$
for the system 
$\cF_C^{p,m}(a,B)$ factors through that of $\pi_1((\C_x^{\times})^m,\dot{x})$
via the isomorphism
$$
\pi_1((\C_x^{\times})^m-\sing(\wt{\mathfrak I}),\dot{x})\xrightarrow{\simeq}
\pi_1((\C_x^{\times})^m,\dot{x})
$$
for a base point $\dot{x}$. Since the differential equation
is has only regular singularities, the space of solutions is generated
by functions of the form $\prod_ix_i^{\beta_i}$. This contradicts to the series
expressions of solutions.
Thus we have Theorem \ref{th:sing-loc}.
\end{proof}
\begin{remark}
Though the ideal $\mathfrak I$ is generated by $\ell_1,\dots,\ell_m$,  
we cannot see whether $\init(\mathfrak I)$ is generated by 
$\init^m(\ell_1),\dots,\init^m(\ell_m)$ or not.   
If $\init(\mathfrak I)$ is generated by 
$\init^m(\ell_1),\dots,\init^m(\ell_m)$, then  
$\init\,(\mathfrak I)$ is independent of the parameters $a,B$,  
and we can remove the conditions 
\eqref{eq:non-integral} and \eqref{eq:non-integral-addition} 
in Theorems \ref{th:rank} and \ref{th:sing-loc}.
\end{remark}

\end{document}